\documentclass[conference, letterpaper,10 pt]{IEEEtran}
\IEEEoverridecommandlockouts

\usepackage{cite}

\ifCLASSINFOpdf
   \usepackage[pdftex]{graphicx}
   \graphicspath{{./}}
   \DeclareGraphicsExtensions{.pdf,.jpeg,.png,.jpg,}
\else
 
\fi

\usepackage{tikz}
\usetikzlibrary{fit}
\usetikzlibrary{arrows,calc,shapes,decorations}
\usepackage{relsize}
\usetikzlibrary{arrows.meta,calc,decorations.markings,math}
\usetikzlibrary{spy}

\usepackage{diagbox}
\usepackage{booktabs}
\usepackage{makecell}
\usepackage{rotating}
\usepackage{multirow}

\usepackage[cmex10]{amsmath}

\usepackage{amsfonts,amssymb}

\usepackage{amsthm}

\newtheorem{theorem}{Theorem}

\newtheorem{corollary}{\textbf{Corollary}}[theorem]
\newtheorem{lemma}{\textbf{Lemma}}

\theoremstyle{definition}

\newtheorem{remark}{\textbf{Remark}}
\newcommand{\norm}[1]{\left\lVert#1\right\rVert} 

\usepackage{accents}

\usepackage{colortbl}

\usepackage[inline]{enumitem}

\usepackage{array}
\newcolumntype{C}[1]{>{\centering\bfseries}m{#1}}
\usepackage{tabularx}

\allowdisplaybreaks

\ifCLASSOPTIONcompsoc
  \usepackage[caption=false,font=normalsize,labelfont=sf,textfont=sf]{subfig}
\else
  \usepackage[caption=false,font=footnotesize]{subfig}
\fi

\begin{document}
\title{All-Against-One Linear-Quadratic\\ Differential Games}

\author{
	\thanks{The authors are with the Department of Electrical and Computer Engineering, University of Central Florida, Orlando, FL. Emails: shahriar@ece.ucf.edu, simaan@ucf.edu, qu@ucf.edu.}
	\thanks{This work was supported in part by the US National Science Foundation under grants ECCS-1308928.}
\IEEEauthorblockN{Shahriar~Talebi,~\textit{Student~Member,~IEEE,}~Marwan~A.~Simaan,~\textit{Life~Fellow,~IEEE,} ~Zhihua~Qu,~\textit{Fellow,~IEEE,}
}
}

\maketitle

\begin{abstract}
All-Against-One (AAO) games are a special class of multi-player games where all players except one are in direct conflict with the remaining player.  In the case of Linear Quadratic Differential (LQD) games, the AAO structure can be used to describe a situation where all players except one are trying to regulate the state of the system (drive it towards the origin) while the opposing player is trying to de-regulate the state (drive it away from the origin). Similar to the standard LQD games, the closed-loop Nash strategies in the AAO LQD games are expressed in terms of the solutions of a set of coupled matrix Riccati differential equations. However, conditions for existence of a solution to these equations are different and more challenging in the AAO case. In this paper, we derive conditions for existence, definiteness, and uniqueness of a solution to these equations as well as conditions for boundedness of the resulting state trajectory to ensure that the opposing player fails in accomplishing its objective. Finally, we consider two options for designing Nash strategies for the players and we illustrate the results with an example of 3-against-1 pursuit-evasion differential game.
\end{abstract}

\section{Introduction}
Many large enterprises consist of a large number of interacting subsystems.  These systems often operate optimally when all subsystems have a harmonious non-conflicting relationship among themselves.  When one subsystem decides to operate in a manner that is not consistent with the others, the operation of the entire enterprise suffers resulting in an adversarial environment that affects not only the behavior of the dissenting subsystem but also possibly the other conforming subsystems as well.  The disruption caused by one dissenting subsystem may cause the remaining subsystems to act as a unified team or it may result in the entire enterprise disintegrating and behaving in a non-cooperative manner within itself.  These types of systems are best analyzed under the general framework of game theory and more specifically using a new structure, which we refer to as All-Against-One (AAO) games.

An example of AAO games, which has recently been considered in the literature, is the multi-pursuer one-evader pursuit evasion game \cite{talebi2017cooperative}.  In this game, a group of pursuers are trying to catch one evader who is trying to escape.  The solution of these types of games involves the development of movement strategies for the pursuers and evader that will conclude with either the evader escaping or being captured.      Except for this multi-pursuers one-evader problem, game theory has mainly been concerned with independent non-cooperative players with typically different objectives \cite{isaacs1965differential, ho1965differential, starr1969nonzero, foley1974class, basar1999dynamic, engwerda2005lq, lambertini2018differential}. The Nash equilibrium has been a very useful concept in defining strategies for such games \cite{nash1951non}.  The AAO structure, however, provides a unique configuration that allows for a more specific analysis of problems where a group of players, all having different but similar objectives, act against one opposing player. 

In this paper, we focus on All-Against-One (AAO) Linear Quadratic Differential (LQD) game.  Non-zero sum LQD games have been studied for decades, beginning with the seminal paper by Starr and Ho \cite{starr1969nonzero}.  An M-player LQD game is described by the system dynamics:
\begin{equation}\label{eqn:sys-gen}
 \dot{x} = A x + \sum_{j=1}^M{B_j u_j}, \quad x(t_0) =x_0;
\end{equation}
and cost functions
\begin{equation}\label{eqn:cost-gen-p}
 J_i = \frac{1}{2} x^\intercal(t_f) S_{if} x(t_f) + \frac{1}{2} \int_{t_0}^{t_f} \big[x^\intercal Q_i x
+ u_i^\intercal R_i u_i \big]dt
\end{equation}
for $i=1,2,...,M$, where $x$ is the state vector and $u_1$ through $u_M$ are the players' control vectors. Matrices $A$, $B_i$, $Q_i$ and $R_i$ are all bounded and of proper dimensions and $S_{if} \geq 0$ and $Q_i \geq 0$ are symmetric positive semi-definite matrices and $R_i > 0$ is symmetric positive definite for $i = 1, 2, ..., M$.  For simplicity of notation and derivations, we have not included the input cross-coupling terms in the cost functions.

The AAO LQD game is described by the same dynamics as in (\ref{eqn:sys-gen}) and cost functions as in (\ref{eqn:cost-gen-p}) except that now $S_{1f} < 0$ and $Q_1 < 0$ are symmetric negative definite matrices while $S_{if} \geq 0$ and $Q_i \geq 0$ for $i = 2, 3, ..., M$ remain as before.  This formulation puts players $2,\dots,M$ directly in conflict with player 1.  That is, while players 2 through $M$ are trying to regulate the state of the system by minimizing its deviation from the origin as in the standard LQD game, player 1 is now trying to de-regulate (or de-stabilize) the system by maximizing the deviation of the state from the origin.   Because of this, existence results of the closed-loop Nash strategies for players in the standard LQD game \cite{papavassilopoulos1979existence, papavassilopoulos1979uniqueness,  freiling1996global, abou2012matrix} do not apply to the AAO games. 

In this paper, we derive conditions for existence, definiteness, and uniqueness of the closed-loop Nash strategies in AAO LQD games as well as boundedness of the resulting state trajectory. The paper is organized as follows: In section \ref{sec:contrib}, we define the problem and derive new conditions for existence and definiteness of the closed-loop Nash strategies as well as sufficient conditions for the exponential boundedness of the resulting state trajectory. One advantage of the AAO structure is that it provides for an additional option for the group of players to consider forming a cooperative team against the opposing player in an attempt to improve their collective performance.  In section \ref{sec:strategies}, we discuss the Team-Nash strategy for the group of players in designing their strategies against the opposing player as an alternative option to the Nash strategy among all players. In section \ref{sec:example}, we present an illustrative example of three-pursuers one-evader, as a three-against-one game that considers different pursuit-evasion scenarios and shows simulation results for both pursuers' strategies when the evader is using a strategy that yields a Nash equilibrium in each case. Concluding remarks are presented in section \ref{sec:conclusion}.

\subsection{Notation}
The real maximum (minimum) eigenvalue of a symmetric $n \times n$ matrix $Q(t)$ at each instant of time $t$ is denoted by $\lambda_{Q}^{max}(t)$ ($\lambda_{Q}^{min}(t)$). 
The maximal (minimal) eigenvalue of the same matrix $Q(t)$ over the interval $[t_0,t_f]$ is a real constant defined as $\bar{\lambda}_{Q} \triangleq \max_{t \in [t_0,t_f]} \lambda_{Q}^{max}(t)$ ($\underline{\lambda}_{Q} \triangleq \min_{t \in [t_0,t_f]} \lambda_{Q}^{min}(t)$). 
The Euclidean norm of a vector $v$ is denoted by $\norm{v}$. The Frobenius norm of a matrix $S$ is defined as $\norm{S}_F =tr\{SS^\intercal\}$ where $tr\{.\}$ is the trace of a matrix. Function $blkdiag\{Q_1,\dots,Q_m\}$ constructs a larger matrix with diagonal blocks consist of matrices $Q_i$. Matrix $I_n$ is the identity matrix of dimension $n$, vector $\mathbf{1}_n$ denotes a vector of dimension $n$ with all entries equal to 1, vector $\mathbf{e}_{i}$ is the standard basis for $\mathbb{R}^n$ in the $i^{th}$ direction and $\otimes$ indicates the Kronecker product. The set of positive (negative) semi-definite $\mathbb{R}^{n \times n}$ matrices is denote by $\mathbb{R}_+^{n \times n}$ ($\mathbb{R}_-^{n \times n}$) and $\prod$ indicates the Cartesian product of sets.\\
We use the notation $\mathcal{I}_{t_f}^\delta$ to denote a subset of $[t_0,t_f]$ defined as $\mathcal{I}_{t_f}^\delta = [t_f-\delta,t_f] \subseteq [t_0,t_f]$ where $\delta$ is arbitrary and $0 < \delta \leq t_f-t_0$.

\section{Characterization of Closed-loop Nash Strategies in AAO LQD Games \label{sec:contrib}}
It is well known \cite{starr1969nonzero} that the closed-loop Nash strategies for the game described by (\ref{eqn:sys-gen}) and (\ref{eqn:cost-gen-p}) are of the form 
\begin{equation}\label{eqn:opt-input}
u_i^* = -R_i^{-1} B_i^\intercal S_i x
\end{equation}
for $i = 1, 2, ..., M$, where $S_i$'s satisfy the following M-coupled differential Riccati equations:
\begin{multline}\label{eqn:Nriccati}
\dot{S_i} + S_i A + A^\intercal S_i + Q_i + S_i H_i S_i \\
- \sum_{j=1}^M \big(S_i H_j S_j + S_j H_j S_i \big) = 0,
\end{multline}
with boundary conditions $S_i(t_f) = S_{if}$, and $H_i = B_i R_i^{-1} B_i^\intercal$. The resulting system trajectory will satisfy
\begin{equation}\label{eqn:sysdyn}
\dot{x} = \overline{A}x
\end{equation}
where
\begin{equation}\label{eqn:Abar}
\overline{A} \triangleq A -  \sum_{j=1}^M H_{j} S_j. 
\end{equation}
The above characterization (and all the subsequent analysis) related to the AAO games can be easily extended if the input cross-coupling terms were presented in the cost functions (\ref{eqn:cost-gen-p}).

\subsection{Definiteness of AAO Solutions}
For the standard LQD game where $S_{if} \geq 0$ and $Q_i \geq 0$ for $i = 1, 2, ..., M$, it is well known \cite{abou2012matrix} that all matrices satisfying (\ref{eqn:Nriccati}) will be positive semidefinite.  That is, $S_i(t) \geq 0$ for $t \in [t_0,t_f]$ and for $i = 1, 2, ..., M$. Because $S_{1f}<0$ and $Q_1<0$ in the AAO games this result is no longer valid for these games.  The following theorem provides the equivalent result for the AAO LQD formulation.
\begin{theorem}
\label{lemma:Sdef} 
For the AAO game, let $\{S_i(t), i=1,\cdots, M\}$ satisfy (\ref{eqn:Nriccati}) over any interval $\mathcal{I}_{t_f}^\delta \subseteq [t_0,t_f]$, then for all $t \in \mathcal{I}_{t_f}^\delta$ we have $S_1(t) < 0$ and $S_i(t) \geq 0$ for $i = 2, 3, ..., M$.
\end{theorem}
\begin{proof}
Let $\lambda_{S_1}^{max}(t)$ and $v(t)$ be the maximum eigenvalue and corresponding unit eigenvector of $S_1(t)$. The matrix $S_1(t)$ is piecewise continuously differentiable and symmetric for $ t \in \mathcal{I}_{t_f}^\delta$ but not necessarily analytic. 
Thus, its eigenvectors (or eigenspaces) are not necessarily continuous (therefore not differentiable) \footnote{An example of such a matrix is presented in [Kato \citenum{kato2013perturbation}, p. 111]}. Continuity (differentiability) of the eigenvectors of $S_1(t)$ requires $S_1(t)$ to be holomorphic (analytic) (see \cite{kato2013perturbation}, Theorem 1.10). However, in our case we only have continuity of $S_1(t)$ and thus we cannot differentiate $v(t)$.
Instead, we can show [see \cite{abou2012matrix}, Theorem 3.6.1] that $\lambda_{S_1}^{max}(t)$ is not only continuous but also differentiable almost everywhere, and at any point that it is differentiable its derivative satisfies
\begin{align}
\frac{d}{dt} \lambda_{S_1}^{max}(t) =& v^\intercal(t) \dot{S_1} v(t) \nonumber \\
=& -v^\intercal(t) \left( S_1 A + A^\intercal S_1 + S_1 H_1 S_1 \right) v(t) \nonumber\\
& + \sum_{j=1}^M \left( v^\intercal(t) S_1 H_j S_j v(t) + v^\intercal(t) S_j H_j S_1 v(t) \right) \nonumber\\
& - v^\intercal(t) Q_1 v(t) \nonumber
\end{align}
which can be simplified as
\begin{multline}
\frac{d}{dt} \lambda_{S_1}^{max}(t) = \lambda_{S_1}^{max}(t) v^\intercal(t) \Big[ {- A - A^\intercal} - \lambda_{S_1}^{max}(t) H_1 \\
+ \sum_{j=1}^M \big( H_j S_j + S_j H_j \big) \Big] v(t) - v^\intercal(t) Q_1 v(t). \nonumber
\end{multline}
Now for sufficiently small $\lambda_{S_1}^{max}(t)$ we have
\begin{equation}\label{eqn:lambdaapprox}
\frac{d}{dt} \lambda_{S_1}^{max}(t) \approx  - v^\intercal(t) Q_1 v(t) > 0
\end{equation}
since $Q_1 < 0$. Given that  $S_1(t_f)=S_{1f} < 0$, we know $\lambda_{S_1}^{max}(t_f) < 0$ and hence by (\ref{eqn:lambdaapprox}) and the continuity of $\lambda_{S_1}^{max}(t)$, the maximal eigenvalue $\lambda_{S_1}^{max}(t) < 0$ for all $t \in \mathcal{I}_{t_f}^\delta$. 
This proves that $\lambda_{S_1}^{max}(t)$ starts from a negative value at $t=t_f$ and stays negative as it evolves backward in time. Noting that $\lambda_{S_1}^{max}(t)$ is the maximum eigenvalue of $S_1(t)$, this proves that $S_1(t) < 0$ for all $t \in \mathcal{I}_{t_f}^\delta$. 
Now, for $i = 2, 3..., M$, let $\overline{A}$ be as defined in (\ref{eqn:Abar}) then (\ref{eqn:Nriccati}) can be written as
\begin{equation}
\dot{S_i} + S_i \overline{A} + \overline{A}^\intercal S_i + Q_i + S_i H_i S_i = 0\label{eqn:Si_dot} \nonumber
\end{equation}
with $S_i(t_f) = S_{if}$. Let $\Psi(t,\tau)$ be the state transition matrix of $-\overline{A}^\intercal(t)$. It is known that for $t, \tau \in \mathcal{I}_{t_f}^\delta$,
\begin{equation}\label{eqn:psidef}
\frac{\partial \Psi(t,\tau)}{\partial t} = -\overline{A}^\intercal(t) \Psi(t,\tau), \quad \Psi(\tau,\tau) = I_n;
\end{equation}
Then $S_i(t)$ will satisfy
\begin{multline}\label{S_isol}
S_i(t) = \Psi(t,t_f) S_i(t_f) \Psi^\intercal(t,t_f) \\
+ \int_t^{t_f} \Psi(t,\tau) [ Q_i + S_i H_i S_i ] \Psi^\intercal(t,\tau) d\tau.
\end{multline}
Now, since $S_i(t_f) \geq 0, \; Q_i \geq 0$ and $R_i >0$ for $i = 2, 3..., M$, it follows from (\ref{S_isol}) that for all $t \in \mathcal{I}_{t_f}^\delta$ we have $S_i(t) \geq 0$ for $i = 2, 3, ..., M$. This completes the proof.
\end{proof}

\subsection{Existence of AAO Solutions}
Sufficient conditions for existence of $S_i(t)$ matrices satisfying (\ref{eqn:Nriccati}) have been derived in \cite{freiling1996global} and \cite{abou2012matrix} only for an LQD game consisting of two players and with all positive definite weight matrices (i.e. $S_{1f}>0$, $S_{2f}>0$, $Q_1>0$ and $Q_2>0$).  Clearly, in the AAO formulation these conditions do not apply since there is one player with negative definite weight matrices ($S_{1f}<0$ and $Q_1<0$). The next theorem provides sufficient conditions for existence of set of $S_i(t)$ matrices satisfying (\ref{eqn:Nriccati}) for an M-player AAO game. We first state the following two Lemmas that are needed in the proof of theorem \ref{theo:totalexistence}.

\begin{lemma}\label{lemma:Y}
	Let $Y(t)$ be a matrix that satisfies
	\begin{equation}
	\dot{Y}(t) \leq - Y(t) A - A^{\intercal} Y(t), \quad Y(t_f) = \mathbf{0}
	\label{eqn:Ydef1}
	\end{equation}
	over the interval $\mathcal{I}_{t_f}^\delta \subseteq [t_0,t_f]$. Then it follows that $Y(t) \geq 0$, for all $ t \in \mathcal{I}_{t_f}^\delta$.
\end{lemma}
\begin{proof}
   Define $Z(t,\tau) = \Phi^\intercal (t,\tau) Y(t) \Phi(t,\tau)$ while $\Phi(t,\tau)$ for all $t, \tau \in \mathcal{I}_{t_f}^\delta$ is defined as follows
	\begin{equation}\nonumber
	\frac{\partial \Phi(t,\tau)}{\partial t} = A \Phi(t,\tau), \quad \Phi(\tau,\tau) = I_n.
	\end{equation}
	Since $Y(t_f) = \mathbf{0}$ and $\Phi(t,\tau)$ is bounded due to $A$ being bounded, then $Z(t_f,\tau) = \mathbf{0}$ for all $\tau \in \mathcal{I}_{t_f}^\delta$. Also,
	\begin{equation}\nonumber
	\frac{\partial Z(t,\tau)}{\partial t} = \Phi^\intercal (t,\tau) \big[ A^\intercal Y + \dot{Y} + Y A\big] \Phi(t,\tau) \leq 0
	\end{equation}
	where $Y$ is a solution of (\ref{eqn:Ydef1}). Now if we define $g(t,\tau,v) = v^\intercal Z(t,\tau) v$ for all $t, \tau \in \mathcal{I}_{t_f}^\delta$ and $v \in \mathbb{R}^{n}$ then $\frac{\partial}{\partial t} g(t,\tau,v) \leq 0$ for all $t \in \mathcal{I}_{t_f}^\delta$. Therefore the mean value theorem yields
	\begin{equation}\nonumber
	{0 \leq g(t_1,\tau,v) - g(t_2,\tau,v)} = v^\intercal [ Z(t_1,\tau) - Z(t_2,\tau) ] v
	\end{equation}
	for all $t_1, t_2 \in \mathcal{I}_{t_f}^\delta$ where $t_1 \leq t_2$, and as a result  $Z(t_1,\tau) \geq Z(t_2,\tau)$. We conclude that $Z(t,\tau) \geq Z(t_f,\tau) = \mathbf{0}$ for all $t, \tau \in \mathcal{I}_{t_f}^\delta$. Finally choosing $\tau = t$ yields $Y(t) \geq 0$ for all $t \in \mathcal{I}_{t_f}^\delta$. 
\end{proof}

Now, for a bounded symmetric matrix $Q$, define the following map from $\prod_1^M \mathbb{R}^{n \times n} $ to $\mathbb{R}^{n \times n}$ as follows
\begin{multline}\label{eqn:RR}
\mathcal{R}_Q(W_1, ... , W_M) \triangleq Q + W_1 H_1 W_1 - \sum_{i=2}^M W_i H_i W_i \\
+ \Big( -W_1 + \sum_{i=2}^M W_i \Big)  \sum_{j=1}^M H_j W_j + \sum_{j=1}^M W_j H_j \Big( -W_1 + \sum_{i=2}^M W_i \Big)
\end{multline}
and let $L_Q(t)$ be the unique solution of following terminal value linear differential equation
\begin{multline}\label{eqn:LQ}
\dot{L}_Q(t) =  - L_Q(t) A - A^{\intercal} L_Q(t) - \Big(Q - Q_1 + \sum_{i=2}^M Q_i \Big)
\end{multline}
with boundary condition $L_Q(t_f) = -S_{1f} + \sum_{i=2}^M S_{if}$. The following lemma provides the condition under which solutions $S_i(t)$ of (\ref{eqn:Nriccati}) stay in the following bounded set:
\begin{equation}\label{eqn:Edef}
\mathcal{E} \triangleq \left \{ S \in \mathbb{R}^{n \times n} | S=S^\intercal, \; \norm{S}_F \leq \overline{l_q} \right \}
\end{equation} 
where $\overline{l_q} = \sup\limits_{t \in [t_0,t_f]} \norm{L_Q(t)}_F$. We also define the space $\mathbb{\chi} \triangleq C\left(\mathcal{I}_{t_f}^\delta, \prod_1^M \mathcal{E} \right)$. Space $\mathbb{\chi}$ is complete under the norm\footnote{Since it is a closed subset of $C\left( \mathcal{I}_{t_f}^\delta, \prod_1^M \mathbb{R}^{n \times n}  \right)$ which is a Banach space under the same norm.}
\begin{equation}\label{eqn:normx}
\norm{(S_1,\dots,S_M)}_\chi \triangleq \sup\limits_{t \in \scalebox{0.9}{$\mathcal{I}_{t_f}^\delta$}} \sum_{j=1}^M \norm{S_j}_F.
\end{equation}

This lemma enables us to establish Theorem \ref{theo:totalexistence} on the existence of a set of solutions $S_i(t)$ of (\ref{eqn:Nriccati}).

\begin{lemma}\label{lemma:Sbound}
For the AAO game,
every solution set \{$S_i(t)$, $i=1,2,...,M$\} satisfying (\ref{eqn:Nriccati}) over any interval $\mathcal{I}_{t_f}^\delta \subseteq [t_0,t_f]$ stays within $\mathcal{E}$ for all $t \in \mathcal{I}_{t_f}^\delta \subseteq [t_0,t_f]$ provided that $\mathcal{R}_Q$ maps $\mathbb{R}_-^{n \times n} \times (\prod_2^M \mathbb{R}_+^{n \times n})$ into $\mathbb{R}_+^{n \times n}$.
\end{lemma}
\begin{proof}
Since $L_Q(t)$ satisfies the linear differential equation (\ref{eqn:LQ}) and matrices $A$, $Q$ and $Q_i$ are bounded, then (by Gronwall's Lemma \cite{perko2013differential} and Theorem 1.1.5 in \cite{abou2012matrix}) $L_Q(t)$ exists for all $t \in [t_0, t_f]$. Therefore, $\overline{l_q}$ exists and the set $\mathcal{E}$ in (\ref{eqn:Edef}) is well-defined and bounded.
Next, we prove that 
\begin{equation}
\label{eqn:Wineq2}
0 < -S_1(t) + \sum_{i=2}^M S_i(t) \leq L_Q(t)
\end{equation}
The left inequality of (\ref{eqn:Wineq2}) follows from Theorem \ref{lemma:Sdef}. Now for the right inequality, define
\begin{equation}\nonumber
Y(t) \triangleq L_Q(t) + S_1(t) - \sum_{i=2}^M S_i(t)
\end{equation}
then after considerable algebraic manipulations (\ref{eqn:Nriccati}), (\ref{eqn:RR}) and (\ref{eqn:LQ}) yield
\begin{equation}\nonumber
\dot{Y}= - Y A - A^{\intercal} Y - \mathcal{R}_Q(S_1, ... , S_M)
\end{equation}
with boundary condition $Y(t_f) = \mathbf{0}$. Thus 
\begin{equation}\nonumber
\dot{Y}\leq-YA-A^{\intercal}Y
\end{equation} 
over $\mathcal{I}_{t_f}^\delta \subseteq [t_0,t_f]$ provided that $\mathcal{R}_Q(S_1, ... , S_M) \geq 0$ which follows by the assumption and Theorem \ref{lemma:Sdef}.
Therefore, according to Lemma \ref{lemma:Y} we conclude $Y(t) \geq 0$ for all $t \in \mathcal{I}_{t_f}^\delta$. Therefore, inequality (\ref{eqn:Wineq2}) is proved for all $t \in \mathcal{I}_{t_f}^\delta$. Since according to Theorem \ref{lemma:Sdef}, $-S_1(t) > 0$ and $S_i(t) \geq 0$ for $i=2,3,\dots,M$, it follows that for $i=1,2,\dots,M$ and for all $t \in \mathcal{I}_{t_f}^\delta$
\begin{align}
\norm{S_i(t)}_F =& tr\{S_i(t) S_i^\intercal(t)\} \nonumber\\
 			\leq & tr\{L_Q(t) L_Q^\intercal(t)\} \nonumber\\
            \leq & \overline{l_q}
\end{align}
This completes the proof.
\end{proof}

\begin{theorem}\label{theo:totalexistence}
For the AAO game, there exists a unique solution set \{$S_i(t)$, $i=1,2,...,M$\} satisfying (\ref{eqn:Nriccati}) for all $t \in [t_0,t_f]$
provided that $\mathcal{R}_Q$ maps $\mathbb{R}_-^{n \times n} \times (\prod_2^M \mathbb{R}_+^{n \times n})$ into $\mathbb{R}_+^{n \times n}$.
\end{theorem}
\begin{proof}
The Riccati differential equations in (\ref{eqn:Nriccati}) are equivalent to the following integral equations for $i=1,2,\dots,M$,
\begin{equation}
S_i(t) = S_{if} + \int_t^{t_f} f_i(\tau,S_1,\dots,S_M) d\tau
\end{equation}
where
\begin{equation}\nonumber
f_i\triangleq S_i A + A^\intercal S_i + Q_i - S_i H_i S_i - \sum_{\substack{j=1, j \neq i}}^M \big(S_i H_j S_j + S_j H_j S_i \big).
\end{equation}
Functions $f_i$ are locally Lipschitz for $i=1,2,\dots,M$ because
\begin{align}
\norm{\Delta f_i}_F =& \norm{f_i(t,S_1,\dots,S_M) - f_i(t,S_1^*,\dots,S_M^*)}_F \nonumber \\
		=&  \lVert(S_i - S_i^*) A + A^\intercal (S_i - S_i^*) \nonumber \\
		& - S_i H_i (S_i - S_i^*)-  (S_i - S_i^*) H_i S_i^* \nonumber \\
		&- \sum_{\substack{j=1, j \neq i}}^M \Big((S_i - S_i^*) H_j S_j + S_j^* H_j (S_i - S_i^*)\nonumber\\
		& + S_i^* H_j (S_j - S_j^*) + (S_j - S_j^*) H_j S_i \Big)\rVert_F \nonumber \\
		\leq& \norm{S_i - S_i^*}_F \Big\{ 2 \norm{A}_F  \nonumber \\
        &+ \overline{h} \sum_{j=1}^M \Big(\norm{S_j}_F + \norm{S_j^*}_F \Big) \Big\} \nonumber\\
		&+ \hspace{-10pt} \sum_{\substack{j=1, j \neq i}}^M \hspace{-7pt} \norm{S_j - S_j^*}_F \Big\{\overline{h}\Big( \norm{S_i^*}_F + \norm{S_i}_F \Big) \Big\} \nonumber\\
		\leq& K_c \sum_{j=1}^M \norm{S_j - S_j^*}_F \label{eqn:flip}
\end{align}
where
\begin{equation}
K_c = 2 \norm{A}_F  + \overline{h} \sum_{j=1}^M \Big(\norm{S_j}_F + \norm{S_j^*}_F \Big)
\end{equation}
and $\overline{h} = \max\limits_{1 \leq i \leq M}\norm{H_i}_F$. 
It follows that, $K_c$ is bounded for all $S_i, S_i^* \in \mathcal{E}$ as follows
\begin{equation}\label{eqn:kcbound}
 \sup\limits_{t  \in \scalebox{0.9}{$\mathcal{I}_{t_f}^\delta$}} K_c \leq 2 \sup\limits_{ t \in [t_0,t_f]} \norm{A}_F  + 2 M \overline{h} \; \overline{l_q} \triangleq \overline{K}_c
\end{equation}
for some $\delta > 0$ to be determined later and denoted as $\widehat{\delta}$ according to the following process.

\textbf{\textit{Process:}} Noting from (\ref{eqn:Edef}) that the boundary condition $S_{if} \in \mathcal{E}$ for $i=1,2\dots,M$.
Next, operator $\mathcal{T} \colon \chi \to \chi$ is defined as follows:
\begin{multline}\label{eqn:defT}
\mathcal{T}(S_1,\dots,S_M)_{(t)} \triangleq \\
\left\lbrace \left( S_{1f} + \int_t^{t_f} f_1 d\tau\right) \; ,\dots, \; \left( S_{Mf} + \int_t^{t_f} f_M d\tau \right)\right\rbrace.
\end{multline}
Operator $\mathcal{T}$ is a contraction mapping for $ \delta = \widehat{\delta} = \frac{1}{2M \overline{K}_c}$, since
\begin{align}
\norm{\Delta \mathcal{T}}_\chi =& \norm{\mathcal{T}(S_1,\dots,S_M)_{(t)} - \mathcal{T}(S_1^*,\dots,S_M^*)_{(t)}}_\chi \nonumber\\
		=&\norm{ \left\lbrace \int_t^{t_f} \Delta f_1 d\tau \; ,\dots, \;  \int_t^{t_f} \Delta f_M d\tau \right\rbrace }_\chi \nonumber\\
					  =& \sup\limits_{ t \in \scalebox{0.9}{$\mathcal{I}_{t_f}^{\widehat{\delta}}$}} \sum_{i=1}^M \norm{\int_t^{t_f} \Delta f_i d\tau}_F \nonumber\\
                      \leq& \sup\limits_{t \in \scalebox{0.9}{$\mathcal{I}_{t_f}^{\widehat{\delta}}$}} \sum_{i=1}^M \int_t^{t_f} \norm{ \Delta f_i }_F d\tau \nonumber\\
                      \leq& {\widehat{\delta}} M \sup\limits_{t \in \scalebox{0.9}{$\mathcal{I}_{t_f}^{\widehat{\delta}}$}} \norm{ \Delta f_i }_F \nonumber\\
                      \leq& {\widehat{\delta}} M \overline{K}_c \norm{(S_1,\dots,S_M) - (S_1^*,\dots,S_M^*)}_\chi \nonumber\\
                      =& \frac{1}{2} \norm{(S_1,\dots,S_M) - (S_1^*,\dots,S_M^*)}_\chi
\end{align}
for all $(S_1,\dots,S_M), (S_1^*,\dots,S_M^*) \in \chi$.
By the contraction mapping theorem \cite{arnold1992ordinary}, there exists a unique solution set \{$S_i(t)$, $i=1,2,...,M$\} of (\ref{eqn:Nriccati}) over the interval $[t_f-\widehat{\delta},t_f]$.
Furthermore, by the assumption and Theorem \ref{lemma:Sdef} it follows that $\mathcal{R}_Q(S_1, ... , S_M)  \geq 0$ over $\mathcal{I}_{t_f}^{\widehat{\delta}}$, and by applying Lemma \ref{lemma:Sbound} on $S_i(t)$ over $\mathcal{I}_{t_f}^{\widehat{\delta}}$ , we conclude that $S_i(t_f-\widehat{\delta}) \in \mathcal{E}$ for $i=1,2,...,M$. 

Now, we extend the local existence property established using the above process to the entire interval $[t_0, t_f]$ as follows.
We repeat the same process for the Riccati equations in (\ref{eqn:Nriccati}) over the interval $\mathcal{I}_{t_f-\widehat{\delta}}^{\widehat{\delta}} = [t_f - 2\widehat{\delta}, t_f-\widehat{\delta}]$ with $S_i(t_f-\widehat{\delta}) \in \mathcal{E}$ as the new boundary condition. By recalling that the Lipschitz condition (\ref{eqn:flip}) holds and defining a contraction mapping as in (\ref{eqn:defT}) over a new space $\chi$ defined over the new interval $\mathcal{I}_{t_f-\widehat{\delta}}^{\widehat{\delta}} =[t_f-2\widehat{\delta},t_f-\widehat{\delta}]$, it can be similarly shown that $S_i(t)$ exist over this interval and that based on Lemma \ref{lemma:Sbound} $S_i(t_f-2\widehat{\delta}) \in \mathcal{E}$.
Finally, this process can be repeated over successive intervals $\mathcal{I}_{t_f-k\widehat{\delta}}^{\widehat{\delta}} =[t_f-k\widehat{\delta},t_f-(k-1)\widehat{\delta}]$ for $k=1,2,\dots$. The initial time $t_0$ will be reached since at every iteration $\widehat{\delta} = \frac{1}{2M \overline{K}_c} > 0$ and $\overline{K}_c$ is a uniform bound on $\mathcal{E}$ as shown in (\ref{eqn:kcbound}). This ends the proof.
\end{proof}

\begin{remark}For the existence of AAO solutions note that:
\begin{enumerate}
\item The purpose of the symmetric matrix $Q$ appearing in (\ref{eqn:LQ}) and Theorem \ref{theo:totalexistence} is to establish the upper bound $L_Q(t)$ described in (\ref{eqn:Edef}).

\item It can be shown that the condition 
$\mathcal{R}_Q \left(\mathbb{R}_-^{n \times n} \times (\prod_2^M \mathbb{R}_+^{n \times n}) \right) \subset \mathbb{R}_+^{n \times n}$
is fulfilled for several sub-classes of AAO games, for which explicit conditions on existence of a solution can be determined in terms of weighting matrices $Q_i$ and $S_{if}$. One such subclass is illustrated in Corollary \ref{coro:exist}.
\end{enumerate}
\end{remark}

Before providing the explicit conditions of existence, we provide the following lemma which is used in the proof of the next corollary and Theorem \ref{theo:capt}.

\begin{lemma}\label{lemma:Ssum}
	For the AAO game, if $Q_1 + \sum_{i=2}^M Q_i > 0$ and $S_{1f} + \sum_{i=2}^M S_{if} > 0$, then for every solution set \{$S_i(t)$, $i=1,2,...,M$\} satisfying (\ref{eqn:Nriccati}) over any interval $\mathcal{I}_{t_f}^\delta \subseteq [t_0,t_f]$, it follows that $P(t) >0,$ for all $t \in \mathcal{I}_{t_f}^\delta$.
\end{lemma}
\begin{proof}
	For every solution set \{$S_i(t)$, $i=1,2,...,M$\} that satisfies (\ref{eqn:Nriccati}), let \begin{equation}
P(t)= S_1(t) + \sum_{i=2}^M S_i(t) \label{eqn:Pdef}
\end{equation} then it would satisfy the following
\begin{equation}\label{eqn:Pdot}
\dot{P} + P \overline{A} + \overline{A}^\intercal P +  Q_1 + \sum_{i=2}^M  Q_i +  S_1 H_1 S_1 + \sum_{i=2}^M S_i H_i S_i = 0;
\end{equation}
with boundary condition $P(t_f) =  S_{1f} + \sum_{i=2}^M S_{if}$.
Then $P(t)$ will satisfy
\begin{multline}\label{Psol}
P(t) = \Psi(t,t_f) P(t_f) \Psi^\intercal(t,t_f) + \int_t^{t_f} \Psi(t,\tau) \Big[  Q_1 \\
+ \sum_{i=2}^M  Q_i +  S_1 H_1 S_1 + \sum_{i=2}^M S_i H_i S_i \Big] \Psi^\intercal(t,\tau) d\tau.
\end{multline}
where $\Psi(t,\tau)$ is as in (\ref{eqn:psidef}). Since matrices $R_i >0$ for $i = 1,2,\dots,M$, this yields $H_i \geq 0$, and from the hypothesis of the lemma and (\ref{Psol}), it follows that $P(t) > 0$ for all $t \in \mathcal{I}_{t_f}^\delta$. This completes the proof.
\end{proof}

\begin{corollary}\label{coro:exist}
Consider an AAO game with diagonal matrices $A$, $B_i$, $S_{if}$, $Q_i$ and $R_i$ for $i=1,2,\dots,M$. If $ Q_1 + \sum_{i=2}^M Q_i > 0$, $ S_{1f} + \sum_{i=2}^M S_{if} > 0$, and $H_1 \leq H_i$ for $i=2,3,\dots,M$, then $\mathcal{R}_Q$ maps $\mathbb{R}_-^{n \times n} \times (\prod_2^M \mathbb{R}_+^{n \times n})$ into $\mathbb{R}_+^{n \times n}$ for any $Q \geq 0$ and thus a unique solution set \{$S_i(t)$, $i=1,2,...,M$\} of (\ref{eqn:Nriccati}) exists for all $t \in [t_0,t_f]$.
\end{corollary}
\begin{proof}
See Appendix \ref{proof:coro:exist}.
\end{proof}

\subsection{Boundedness of the System Trajectory}
The next theorem provides a sufficient condition for the system trajectory to be exponentially bounded. Specifically, if the final time is large enough, then the state vector of the system is guaranteed to end up in a small hyper-ball. Essentially, this theorem provides conditions that guarantee that opposing player cannot accomplish its goal of de-stabilizing the system.

\begin{theorem}\label{theo:capt}
For an AAO game with a solution set $\{S_i(t), i=1,\cdots, M\}$ satisfying (\ref{eqn:Nriccati}), if $Q_1 + \sum_{i=2}^M Q_i > 0$ and $S_{1f} + \sum_{i=2}^M S_{if} > 0$, then the state trajectory in (\ref{eqn:sysdyn}) will be exponentially bounded. In particular, for any small $r>0$, it follows $\|x(t_f)\|\leq r$ provided that
\begin{equation}
t_f-t_0 \geq \frac{\overline{\lambda}_{P}}
{\underline{\lambda}_{\widehat{Q}}} \Bigg\{ 2
\ln\left(\frac{\|x_0\|}{r}\right)
+ \ln \left(\frac{\overline{\lambda}_{P}}{\underline{\lambda}_{P}} \right)
\Bigg\}.\nonumber
\end{equation}
\end{theorem}
\begin{proof}
Let
\begin{equation}
V (t)= x^\intercal(t) P(t) x(t),
\end{equation}
where the state vector $x(t)$ is defined in (\ref{eqn:sys-gen}) and $P(t)$ is as in (\ref{eqn:Pdef}). Matrix $P(t)$ is positive definite for all $t$ based on Lemma \ref{lemma:Ssum}. Then, using (\ref{eqn:sysdyn}) and (\ref{eqn:Pdot}) it follows that
\begin{equation}\label{eqn:vdot}
\dot{V}(t) 
= - x^\intercal \left( \sum_{i=1}^M  \left( Q_i + S_i H_i S_i \right) \right) x 
\end{equation}
Since, $Q_1 + \sum_{i=2}^M Q_i > 0$ and $H_i \geq 0$ for $i=1,2,\dots,M$ it follows that $\dot{V}(t) < 0$ which yields a decaying dynamics for the system. Now if we define $\widehat{Q} = \sum_{i=1}^M  Q_i$, then it follows that
\begin{equation}
\dot{V} \leq - x^\intercal \widehat{Q} x \leq  - \frac{\underline{\lambda}_{\widehat{Q}}}{\overline{\lambda}_{P}}
 V \nonumber 
\end{equation}
Then one can conclude that
\begin{equation}
V(t) \leq  V(t_0) \exp\left\{-\frac{\underline{\lambda}_{\widehat{Q}}}{\overline{\lambda}_{P}}(t-t_0)\right\} \nonumber
\end{equation}
which is an exponential decaying bound for the system. Also, we can conclude that
\begin{equation}
\left(\frac{\norm{x(t)}}{\norm{x_0}}\right)^2 \leq \frac{\overline{\lambda}_{P}}{\underline{\lambda}_{P}} \exp\left\{-\frac{\underline{\lambda}_{\widehat{Q}}}{\overline{\lambda}_{P}} (t-t_0)\right\}. \nonumber
\end{equation}
Finally, by rearranging the terms, in order to achieve $\norm{x(t_f)} \leq r$, $t_f-t_0$ should be greater than the expression provided in the theorem. This ends the proof.
\end{proof}

\section{Two Options for the Group of Players \label{sec:strategies}}
The AAO game offers several possible options for the $M-1$ players in the group.  In this paper we consider only two as illustrated in Figure \ref{fig:prob-total}.  The first option is for the players in the group to ignore the fact that they have similar cost functions and decide to design Nash strategies among themselves as well as the opposing player \cite{starr1969nonzero}.  The second option, which intuitively appears to be the better choice, is for the players in the group to form a team and cooperate among themselves in designing a collective Nash strategy against the remaining player \cite{liu2004noninferior}.  However, being restricted to operate within the framework of the team, although beneficial to the entire group, may or may not be beneficial to each player, and as a result may be beneficial to the opposing player. We will refer to the first option as the Nash solution and to the second as the Team-Nash solution. The Nash strategies of the first option are determined as described in (\ref{eqn:opt-input}) and (\ref{eqn:Nriccati}).  The Team-Nash strategies of the second option are determined by first forming a team cost function as a convex combination of the individual cost functions of the $M-1$ players in the group.  That is 
\begin{equation}
J_T=\sum_{i=1}^{M-1} \alpha_i J_{i+1}
\end{equation}
where $\alpha_i>0$, $\sum_{i=1}^{M-1} \alpha_i=1$ and $J_{i}$ for $i = 2, ..., M$ as they are defined in (\ref{eqn:cost-gen-p}).
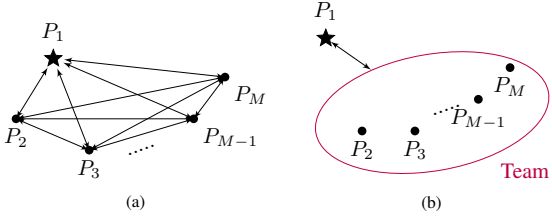
\begin{figure}[!pt]
	\centering
	\resizebox{0.22\textwidth}{!}{
		\begin{subfloat}[]{
				\centering
				\begin{tikzpicture}[
				scale=1.7,
				axis/.style={very thick, ->, >=stealth', line join=miter},
				important line/.style={thick}, dashed line/.style={dashed, thin},
				every node/.style={color=black},
				dot/.style={circle,fill=black,minimum size=4pt,inner sep=0pt,
					outer sep=-1pt},
				]
				\node[star, label=above:$P_1$, star points=5, star point ratio=2.25, draw,inner sep=1.3pt,anchor=outer point 3,fill=black] at (0.8,1.5) (E1) {};
				\node[dot,label=below:$P_2$,scale=1] at (0.5,1) (P1) {};
				\node[dot,label=below:$P_3$,scale=1] at (1.2,.7) (P2) {};
				\node[dot,label=below right:$P_{M-1}$,scale=1] at (2.2,1) (P3) {};
				\node[dot,label=below right:$P_{M}$,scale=1] at (2.5,1.4) (P4) {};  
				\node[rotate= 20] at (1.7,0.7){$.....$};
				\node[circle,scale=1] at (E1) (G) {};
				\draw[latex'-latex',  black, >=stealth']
				($(P1)$) -- (G)
				node[sloped, above, midway] {};
				\draw[latex'-latex',  black, >=stealth']
				($(P2)$) -- (G)
				node[sloped,above, midway] {};
				\draw[latex'-latex',  black, >=stealth']
				($(P3)$) -- (G)
				node[sloped,above, midway] {};
				\draw[latex'-latex',  black, >=stealth']
				($(P4)$) -- (G)
				node[sloped,above, midway] {};   
				\draw[latex'-latex',  black, >=stealth']
				($(P1)$) -- ($(P2)$)
				node[sloped, above, midway] {};
				\draw[latex'-latex',  black, >=stealth']
				($(P1)$) -- ($(P3)$)
				node[sloped, above, midway] {};
				\draw[latex'-latex',  black, >=stealth']
				($(P1)$) -- ($(P4)$)
				node[sloped, above, midway] {};
				\draw[latex'-latex',  black, >=stealth']
				($(P2)$) -- ($(P3)$)
				node[sloped, above, midway] {};
				\draw[latex'-latex',  black, >=stealth']
				($(P2)$) -- ($(P4)$)
				node[sloped, above, midway] {};
				\draw[latex'-latex',  black, >=stealth']
				($(P3)$) -- ($(P4)$)
				node[sloped, above, midway] {};
				\end{tikzpicture}
				\label{fig:prob-Nash}
			}
	\end{subfloat}}
	\resizebox{0.2\textwidth}{!}{
		\begin{subfloat}[]{
				\centering
				\begin{tikzpicture}[
				scale=1.7,
				axis/.style={very thick, ->, >=stealth', line join=miter},
				important line/.style={thick}, dashed line/.style={dashed, thin},
				every node/.style={color=black},
				dot/.style={circle,fill=black,minimum size=4pt,inner sep=0pt,
					outer sep=-1pt},
				]
				\node[star, label=above:$P_1$, star points=5, star point ratio=2.25, draw,inner sep=1.3pt,anchor=outer point 3,fill=black] at (0.5,1.5) (E1) {};
				\node[dot,label=below:$P_2$,scale=1] at (0.9,0.7) (P1) {};
				\node[dot,label=below:$P_3$,scale=1] at (1.4,.7) (P2) {};
				\node[dot,label=below:$P_{M-1}$,scale=1] at (2,1) (P3) {};
				\node[dot,label=below:$P_{M}$,scale=1] at (2.3,1.3) (P4) {};  
				\node[rotate= 20] at (1.7,0.9){$.....$};
				\node[draw, label=below right:{\color{purple}Team}, ellipse, rotate =10 ,fit={ (P1) (P2) (P4) }, xshift=-0.1cm, yshift=-0.2cm, color=purple] (C1) {};
				\draw[latex'-latex'] (C1) -- node[above,rotate=90] {} (E1);
				\end{tikzpicture}
				\label{fig:prob-team}
			}
	\end{subfloat}}
	\caption{All-against-one game \protect\subref{fig:prob-Nash} Nash Strategies: Nash among all players, \protect\subref{fig:prob-team} Team-Nash Strategies: Nash between Player 1 and team of Players 2,3,...,M.}
	\label{fig:prob-total}
\end{figure}
The weight $\alpha_i$ can be viewed as the relative contribution of Player $i+1$ to the team.
So the cost function for the team of players becomes
\begin{equation}\label{eqn:Jp-team}
J_T = \frac{1}{2} x^\intercal(t_f) S_{T_f} x(t_f) + \frac{1}{2} \int_{t_0}^{t_f} \left[x^\intercal Q_T x
+ u_T^\intercal R_T u_T \right]dt
\end{equation}
where 
\begin{align} 
S_{T_f} &= \sum_{i=1}^{M-1} \alpha_i S_{(i+1)f},\\
Q_T &= \sum_{i=1}^{M-1} \alpha_i Q_{i+1}, \\
R_T &= blkdiag \{\alpha_1 R_{2}, \; \alpha_2 R_{3},\;\dots, \;\alpha_{M-1} R_{M}\}.
\end{align}
The opposing player's cost function stays the same as $J_1$.
Also, the control inputs of the players in the group are combined as one vector but it will be reallocated to each player later at the implementation stage.
Consequently, the system dynamics of (\ref{eqn:sys-gen}) can be reorganized as follows
\begin{equation}\label{eqn:sys-team}
\dot{x} = A x + B_1 u_1 + B_T u_T
\end{equation}
where $ u_T = [u_2^\intercal, u_3^\intercal, ... , u_M^\intercal]^\intercal$  is the team control vector. 
	
The system dynamics (\ref{eqn:sys-team}) and cost functions $J_T$ and $J_1$ form a 2-player LQD game.
The closed-loop Nash strategies can be obtained from (\ref{eqn:opt-input}) as follows
\begin{eqnarray}
\begin{cases}
u_1^* &= -R_1^{-1} B_1^\intercal S_1 x \label{eqn:opt-input-2e}\\
u_T^* &= -R_T^{-1} B_T^\intercal S_T x \label{eqn:opt-input-2p}
\end{cases}
\end{eqnarray}
where the matrices $S_T$ and $S_1$ are the solutions to the following coupled differential Riccati equations
\begin{multline}
\dot{S_1} + S_1 A + A^\intercal S_1 + Q_1 - S_1 H_1 S_1 \\
 - S_1 H_T S_T - S_T H_T S_1 = 0,
\end{multline}
\begin{multline}
\dot{S_T} + S_T A + A^\intercal S_T + Q_T - S_T H_T S_T  \\*
- S_T H_1 S_1 - S_1 H_1 S_T = 0,
\end{multline}
where $H_T = B_T R_T^{-1} B_T^\intercal$ and with boundary conditions $S_1(t_f) = S_{1f}$ and $S_T(t_f) = S_{T_f}$.
Note that assigning different values of $\alpha_i$ will yield different teams with different strengths and weaknesses, and correspondingly different Nash solutions. 
Also, note that even though this game is solved as a 2-player LQD game, upon implementation the controls are reallocated to each player and the game remains as a multi-player game.

\section{Illustrative Example \label{sec:example}}
To illustrate the AAO framework, we consider a simple 3-against-1 pursuit evasion game on a planar surface.  The evader, player 1, is trying to escape from three pursuers (players 2, 3 and 4) by maximizing its weighted distance from them while the pursuers are trying to capture the evader by minimizing these distances.  We consider a capture radius of 0.1, meaning that if the evader is within a distance of 0.1 from any of the pursuers it is considered as captured.  The dynamics of the system follows (\ref{eqn:sys-gen}). The state vector $x \in \mathbb{R}^6$ is the combined three displacement vectors of dimension 2 each, which connects the evader E to pursuers P1, P2 and P3, respectively. To control the velocities of the players on the plane, the open-loop dynamic matrix is assumed to be $A=0$. The input matrix for the evader is $B_1 = \mathbf{1}_3 \otimes I_2$ and for the pursuers are $B_{i} = -  \mathbf{e}_{i-1} \otimes I_2$ for $i=2,3,4$. Also, the combined input matrix of the pursuers as a team can be written as $B_T = -I_3 \otimes I_2 $. 
We consider a final time $t_f=10$ (which is relatively long for the given system dynamics). Each player tries to minimize the corresponding cost function as stated in (\ref{eqn:cost-gen-p}). We assume that the evader's parameters in its cost function are as follows
\begin{equation}
	S_{1f} = -18 I_3 \otimes I_2, \quad 
	Q_1 = -6 I_3 \otimes I_2, \quad
	R_1 = I_2 \nonumber
\end{equation}
and the pursuers' parameters in their cost functions are as follows
\begin{equation}
S_{2f} = I_3 \otimes I_2,  \quad
Q_2 = 0.5 I_3 \otimes I_2, \quad
R_2 = 150 I_2, \nonumber
\end{equation}
\begin{equation}
S_{3f} = I_3 \otimes I_2,  \quad
Q_3 = 0.5 I_3 \otimes I_2, \quad
R_3 = 150 I_2, \nonumber
\end{equation}
\begin{equation}
S_{4f} = 16.25 I_3 \otimes I_2,  \;
Q_4 = 5.25 I_3 \otimes I_2, \;
R_4 = 150 I_2. \nonumber
\end{equation}

\begin{figure}[!hp]
	\begin{subfloat}[]{
			\resizebox{!}{0.29\textheight}{%
				\includegraphics[height=4in]{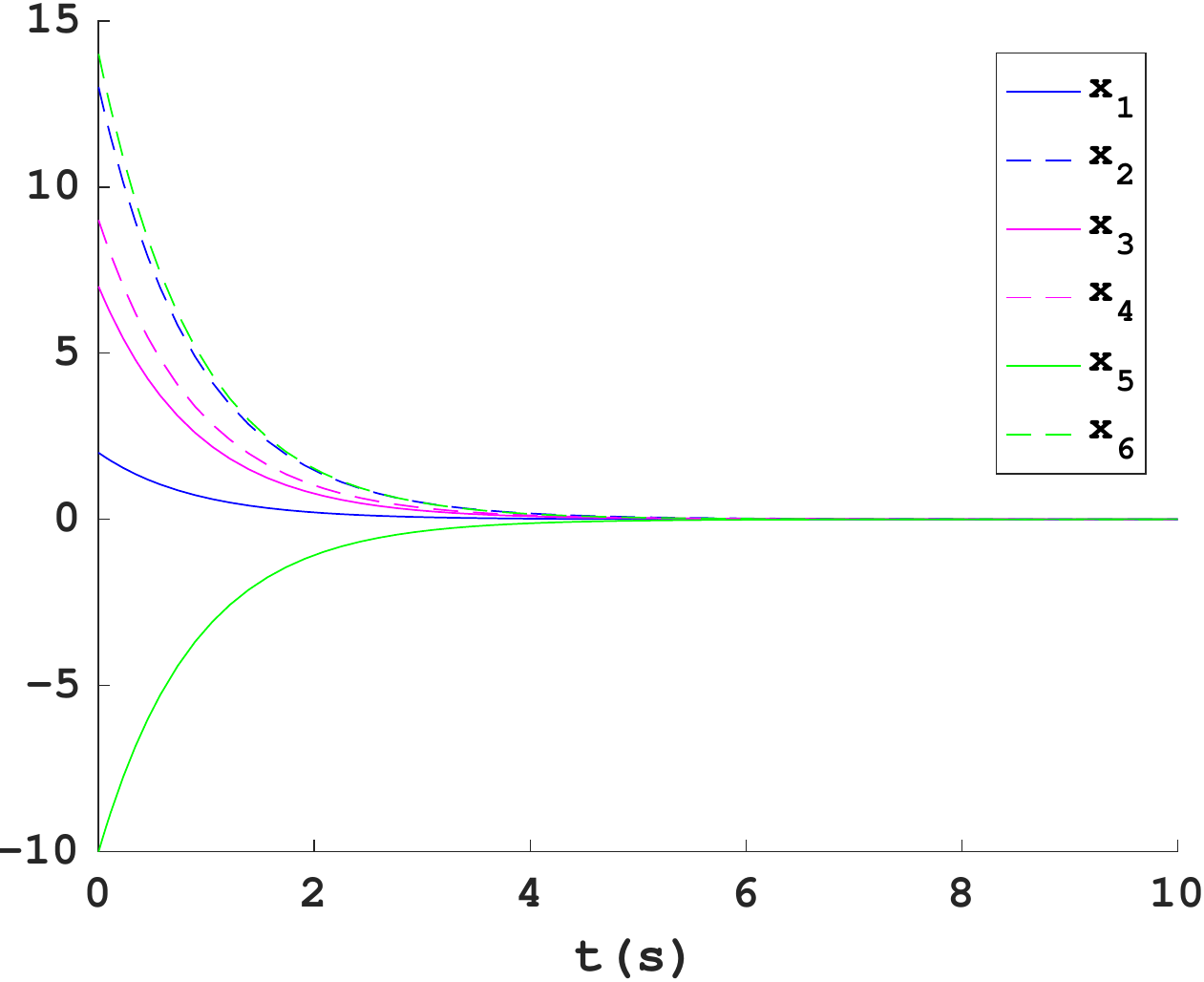}
				\label{fig:expnash}}}
	\end{subfloat}
	\begin{subfloat}[]{
			\resizebox{!}{0.29\textheight}{%
				\includegraphics[width=3.4in]{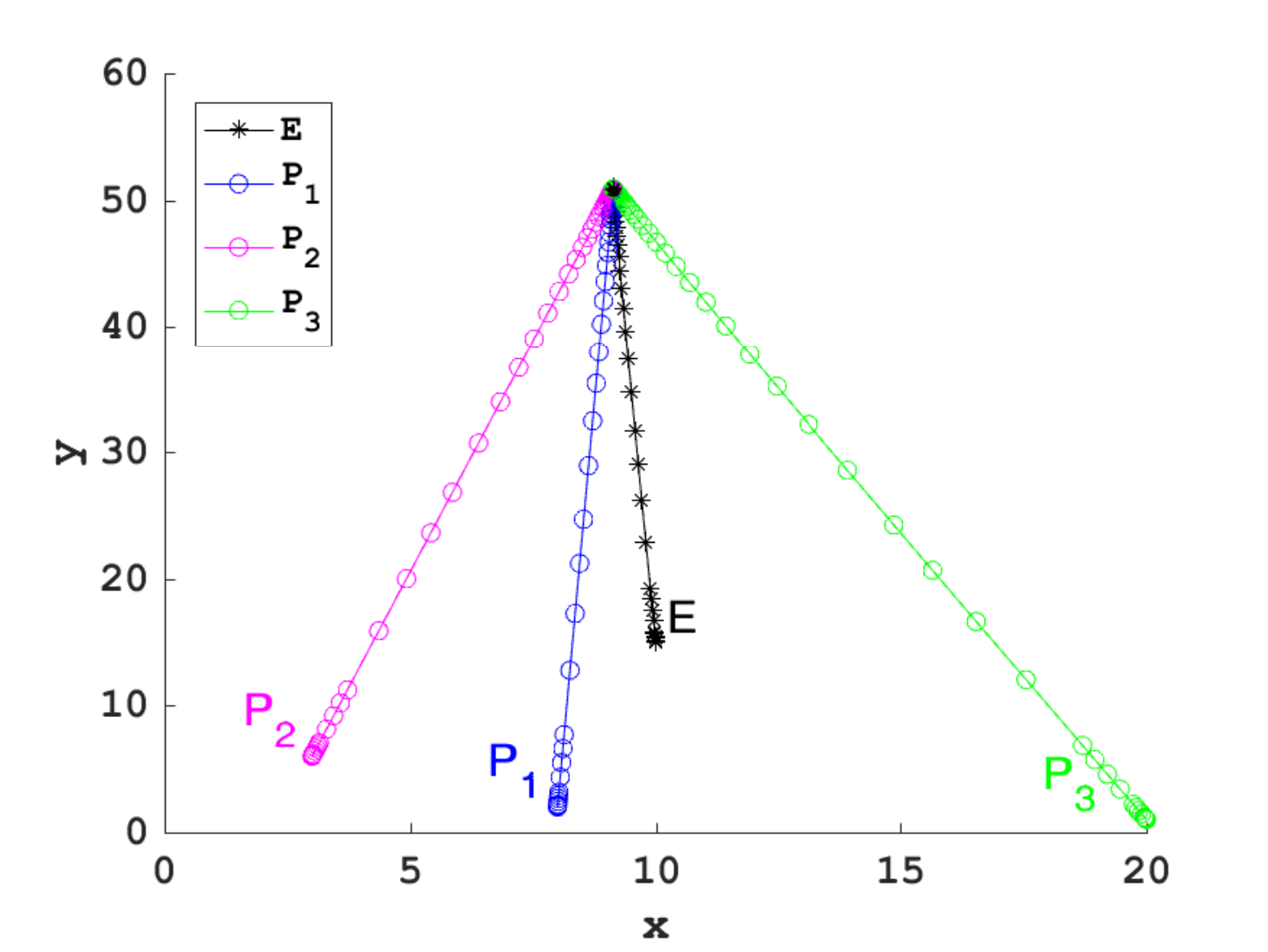}
				\label{fig:expnash2}}}
	\end{subfloat}
	\begin{subfloat}[]{
			\resizebox{!}{0.29\textheight}{%
				\includegraphics[width=3.4in]{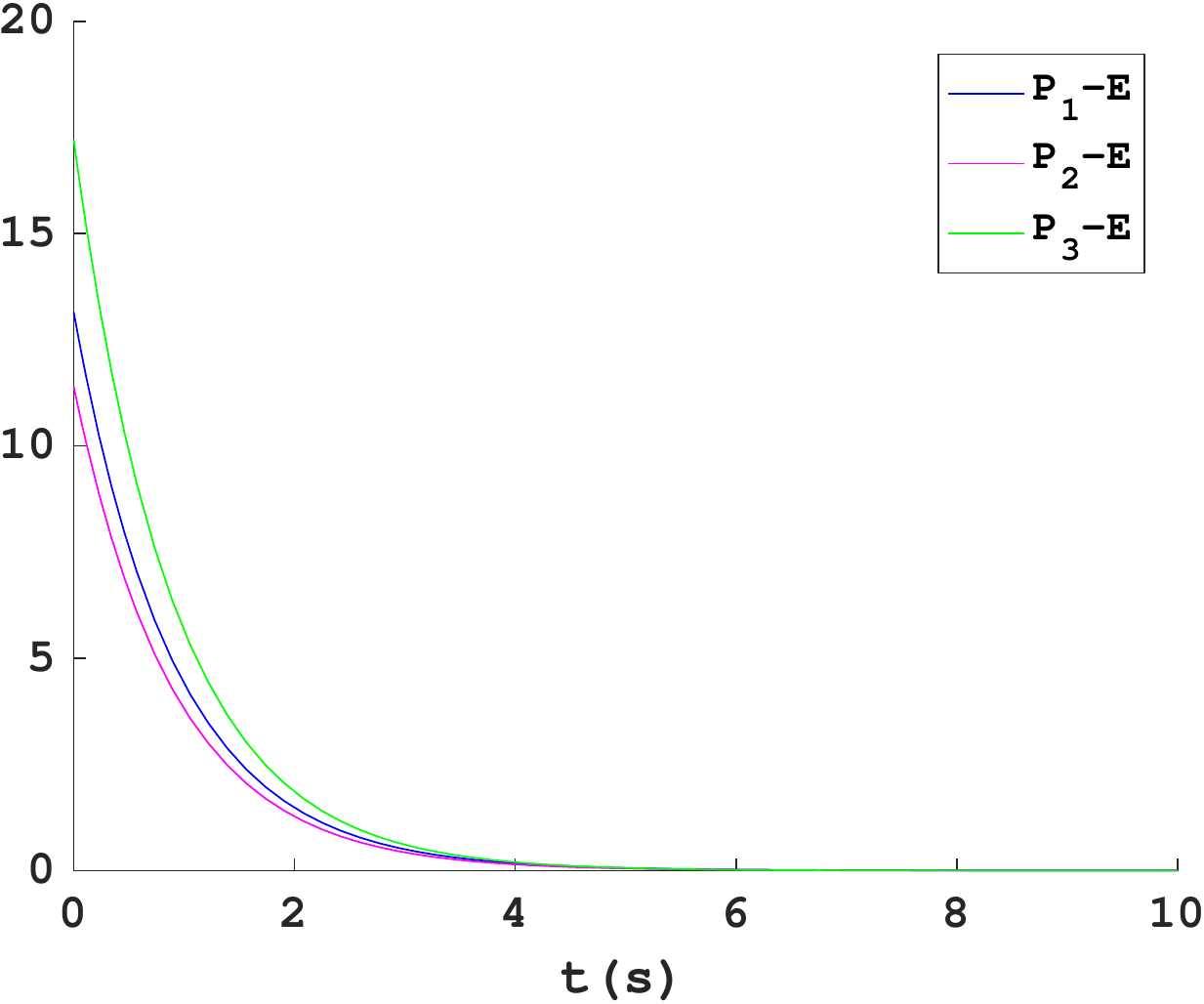}
				\label{fig:expnash3}}}
	\end{subfloat}	
	\caption{Example - Nash Strategies: \protect\subref{fig:expnash} state variables versus time, \protect\subref{fig:expnash2} x-y coordinates of players' trajectories and \protect\subref{fig:expnash3} distances between the pursuers P1, P2 and P3, and the evader E.}
	\label{fig:simNash}
\end{figure}
\begin{figure}[!hp]
	\begin{subfloat}[]{
			\resizebox{!}{0.29\textheight}{%
				\includegraphics[height=3.4in]{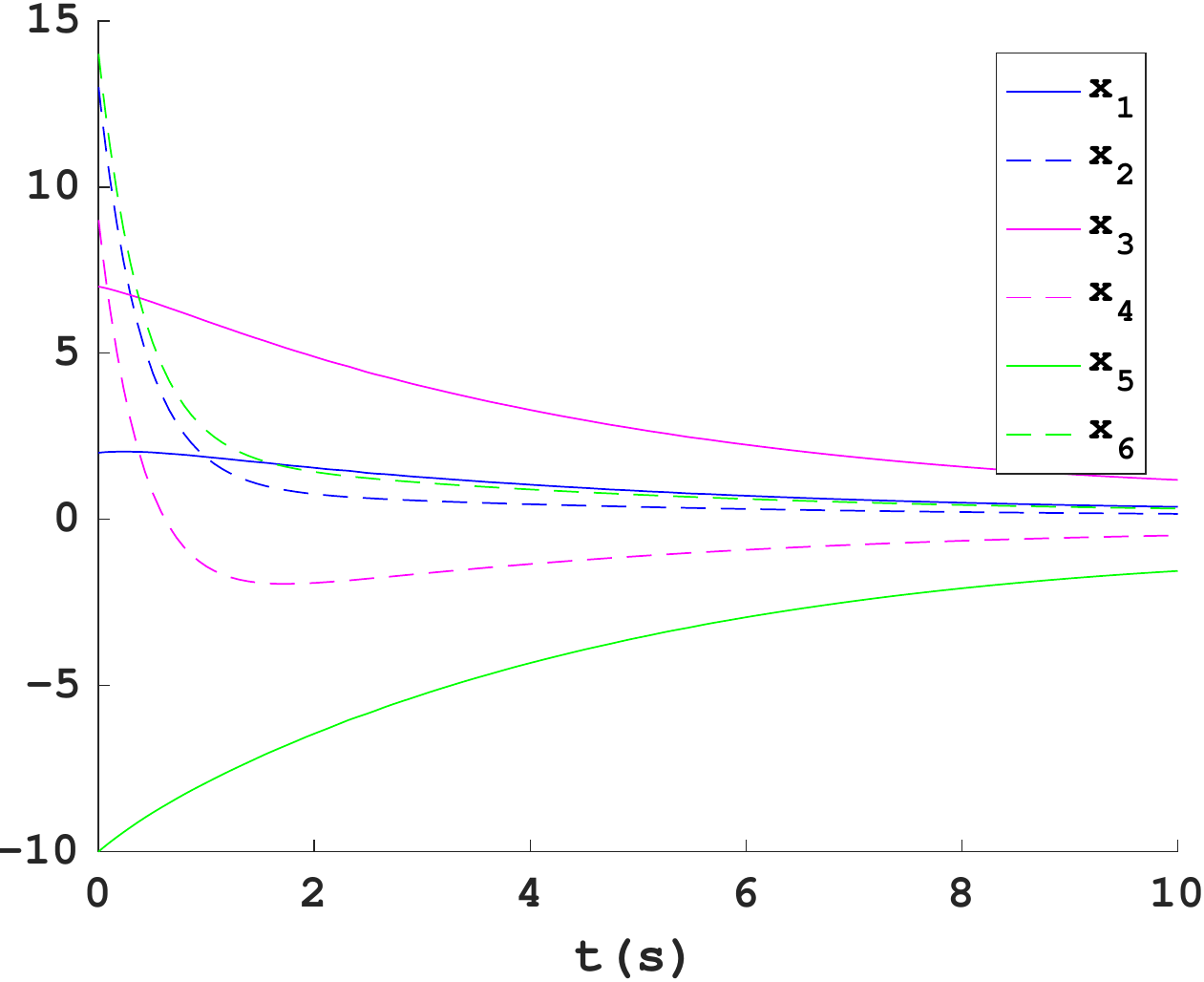}
				\label{fig:expteam}}}
	\end{subfloat}
	\begin{subfloat}[]{
			\resizebox{!}{0.29\textheight}{%
				\begin{tikzpicture}[spy using outlines={ellipse,yellow,magnification=2,size=3cm, height=1.5cm, connect spies}]
				\node {\includegraphics[interpolate=true,width=3.4in]{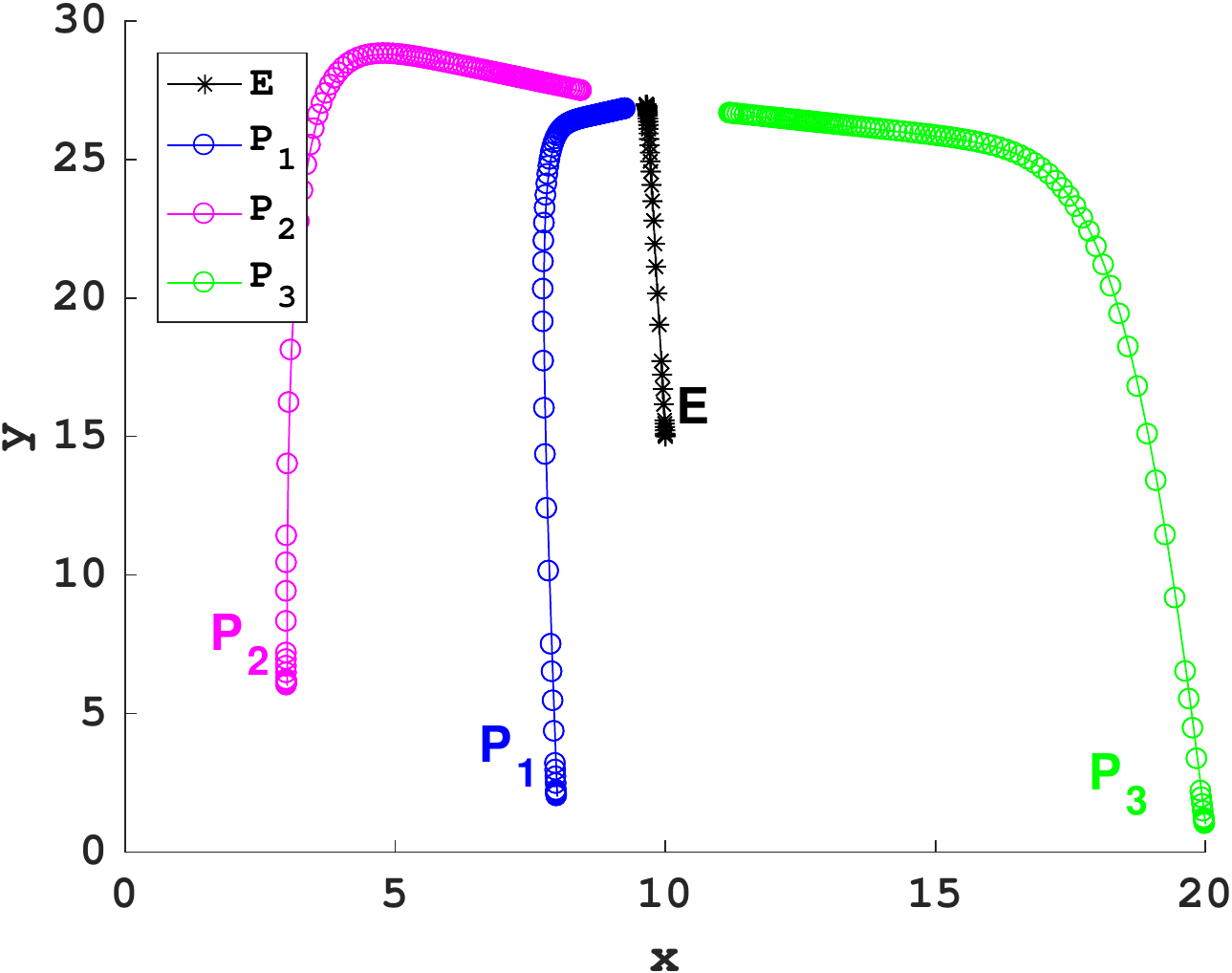}};
				\spy[black] on (0.2,2.7) in node [left] at (2.9,-0.6);
				\end{tikzpicture}
				\label{fig:expteam2}}}
	\end{subfloat}
	\begin{subfloat}[]{
			\resizebox{!}{0.29\textheight}{%
				\includegraphics[width=3.4in]{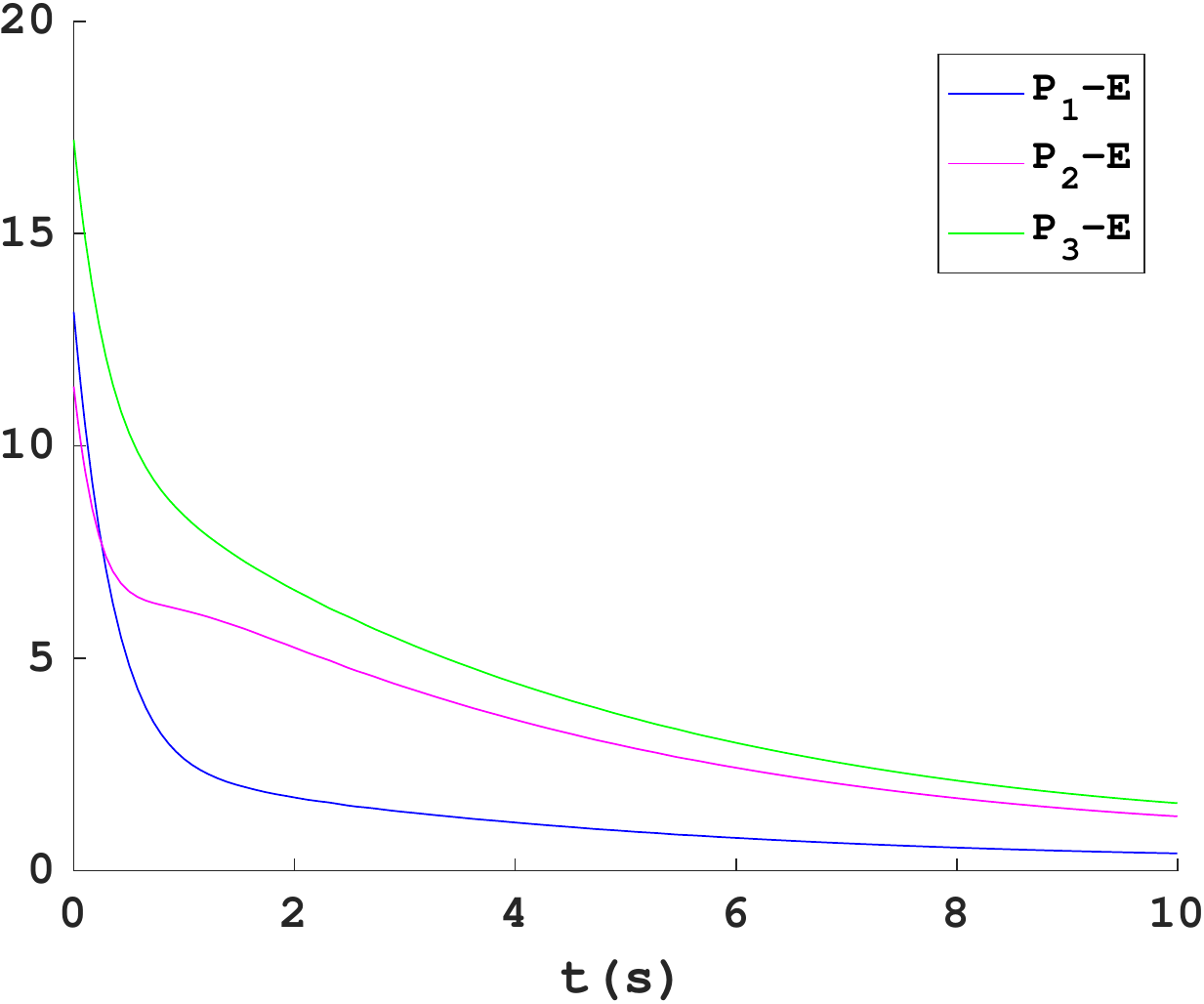}
				\label{fig:expteam3}}}
	\end{subfloat}
	\caption{Example - Team-Nash Strategies: \protect\subref{fig:expteam} state variables versus time, \protect\subref{fig:expteam2} x-y coordinates of players' trajectories and \protect\subref{fig:expteam3} distances between the pursuers P1, P2 and P3, and the evader E.}
	\label{fig:simNashteam}
\end{figure}

Additionally, for the Team-Nash strategies we arbitrarily chose $[\alpha_1 \; \alpha_2 \; \alpha_3] = [0.\overline{3} \; 0.\overline{3} \;  0.\overline{3}]$ representing equal contributions by all pursuers. The initial positions of the players are such that the initial state vector is $x_0 = \begin{bmatrix}
2 & 13 & 7 & 9 & -10 & 14
\end{bmatrix}^\intercal$.
The resulting state variables (which are the combined three displacement vectors that connect the three pursuers to the evader) for the Nash and Team-Nash strategies are plotted vs time in Figures \ref{fig:expnash} and \ref{fig:expteam} respectively. The x-y coordinates of the players' movement for both strategies are shown in Figures \ref{fig:expnash2} and \ref{fig:expteam2} respectively. Finally, the distances of pursuers to the evader for the Nash and Team-Nash strategies are plotted vs time in Figures \ref{fig:expnash3} and \ref{fig:expteam3} respectively.

For the Nash strategies in Figure \ref{fig:simNash}, the conditions of Theorem \ref{theo:capt} are satisfied since
\begin{equation}\label{eqn:cc1}
	 Q_1 + \sum_{j=2}^4 Q_j =  (-6 + 0.5 + 0.5 + 5.25) I_3 \otimes I_2  > 0 \nonumber
\end{equation}
\begin{equation}\label{eqn:cc2}
S_{1f} + \sum_{j=2}^4 S_{jf} = (-18 + 1 + 1 + 16.25) I_3 \otimes I_2 > 0  \nonumber
\end{equation}

Consequently, according to Theorem \ref{theo:capt} the exponential boundedness of the system is guaranteed regardless of the initial position of the players. For the given $t_f$ the pursuers are able to capture the evader before final time of the game.
For the Team-Nash strategies in Figure \ref{fig:simNashteam}, the conditions of Theorem \ref{theo:capt} are not satisfied since
\begin{equation}
Q_1 + Q_T =  (-6 + 2.08) I_3 \otimes I_2  < 0 \nonumber
\end{equation}
\begin{equation}
S_{1f} + S_{T_f} = (-18 + 6.08) I_3 \otimes I_2 < 0.  \nonumber 
\end{equation}
Consequently, the exponential boundedness of system is not guaranteed. In this case as can be seen in Figures \ref{fig:simNashteam}, the evader was able to keep the state vector of the system away from the origin and as a result escape. 
In terms of accomplishing their objectives, the pursuers appear to have done a better job in using the Nash strategies compared to the Team-Nash strategies. As mentioned earlier this may be due to the fact that in the Nash strategies the pursuers are completely free to act independently while in the Team-Nash they are constraint to act cooperatively within the team structure. 

\begin{table}[!ht]
	\caption{Pursuers Distances to the Evader for Different Final Time}
	\label{Table:sim8}
	\centering
	\normalsize
	\showboxdepth=\maxdimen
	\showboxbreadth=\maxdimen
	\begin{tabular}{c| c| c | *{7}{c|}} 
		\toprule
		\multicolumn{3}{c}{\scriptsize $t_f=$} & \scriptsize 2.00 &\scriptsize 4.00 &\scriptsize 6.00 &\scriptsize 8.00 &\scriptsize 10.00 &\scriptsize 12.00 &\scriptsize 14.00 \\ 
		\hline
		\\[-1em]
		\multirow{6}{*}{\begin{turn}{90} \scriptsize Pursuers' Distances to Evader at $t_f$ \end{turn}}&\multirow{4}{*}{\begin{turn}{90} \scriptsize Nash $\; \;$ \end{turn}}& \scriptsize {P1}& \scriptsize 8.25 &\scriptsize 1.42 &\scriptsize 0.25 &\scriptsize 0.03 &\scriptsize 0.00 &\scriptsize 0.00 &\scriptsize 0.00  \\  [0.2ex]
		\cmidrule{3-10}
		& &{\scriptsize {P2} }&\scriptsize 9.59 &\scriptsize 2.22 &\scriptsize 0.40 &\scriptsize 0.05 &\scriptsize 0.00 &\scriptsize 0.00 &\scriptsize 0.00 \\  [0.2ex]
		\cmidrule{3-10}
		& &{\scriptsize {P3} }&\scriptsize 7.19 &\scriptsize 1.29 &\scriptsize 0.16 &\scriptsize 0.02 &\scriptsize 0.00 &\scriptsize 0.00 &\scriptsize 0.00 \\  [0.2ex]
		\cmidrule{2-10}
		&\multirow{4}{*}{\begin{turn}{90} \scriptsize Team-Nash $\;$ \end{turn}}& \scriptsize {P1}&\scriptsize 1.90 &\scriptsize 1.34 &\scriptsize 0.91 &\scriptsize 0.61 &\scriptsize 0.41 &\scriptsize 0.27 &\scriptsize 0.18 \\ [0.2ex]
		\cmidrule{3-10}	
		& &{\scriptsize {P2} }&\scriptsize 5.94 &\scriptsize 4.18 &\scriptsize 2.85 &\scriptsize 1.92 &\scriptsize 1.28 &\scriptsize 0.85 &\scriptsize 0.57 \\  [0.2ex]
		\cmidrule{3-10}
		& &{\scriptsize {P3} }&\scriptsize 7.39 &\scriptsize 5.21 &\scriptsize 3.55 &\scriptsize 2.38 &\scriptsize 1.59 &\scriptsize 1.06 &\scriptsize 0.70 \\  [0.2ex]
		\cmidrule{3-10}
		\bottomrule%
	\end{tabular}%
\end{table}%

Table \ref{Table:sim8} shows the distances between the pursuers and the evader for different values of $t_f$ increasing from 2 to 14 and for the two strategy options discussed. As it is clear from the table for the Nash strategy, the evader could have been captured if the final time were $t_f=8$ instead of $t_f=10$ since its distance from all pursuers are within the capture radius 0.1 at $t_f=8$. On the other hand for the Team-Nash strategy, it is clear that even for a final time $t_f =14$ the pursuers are still further away from the evader and none of them is within the capture radius.
However, if $t_f$ were extended to $18$, then the simulation would show the distance between the pursuer (P1) and the evader at this final time is $0.08$ which is within the capture radius and thus the evader will be captured.

\section{Conclusion \label{sec:conclusion}}
The current literature on linear quadratic non-zero-sum differential games has emphasized a structure in which all players are constrained to have non-negative definite state weight matrices in their cost functions.  This structure does not consider games where one or more players have cost functions that do not satisfy this constraint.  In this paper, we considered games where one of the players has an cost function with negative definite state weight matrices making this player in direct conflict with the remaining players.  These types of games are very useful in analyzing conflict situations in which one player wants to maximize the deviation of the state from the origin (i.e. de-stabilize the system) while the others are trying to regulate the state in the standard sense.  We referred to these types of games as All-Against-One.  In this paper, we derived new conditions that guarantee existence, definiteness and uniqueness of the closed-loop Nash strategies as well as boundedness of the resulting state trajectory. We also considered the Nash and Team-Nash strategies as possible options for the group of the players.  As an illustrative example, we simulated a three-pursuer one-evader pursuit evasion game and determined both close-loop Nash and Team-Nash strategies for the players and resulting trajectories. We also applied our results to derive conditions for which the evader will be captured under each case. In the case of the Nash strategies, we also showed that as the final time increased the distances between the pursuers and the evader became smaller, making it more difficult for the evader to escape.

\bibliographystyle{IEEEtran}
\bibliography{Bibtex}

\appendices

\vspace{0.5cm}
\section{Proof of Corollary \ref{coro:exist} \label{proof:coro:exist}}
First, we prove that $\mathcal{R}_Q$ maps $\mathbb{R}_-^{n \times n} \times (\prod_2^M \mathbb{R}_+^{n \times n})$ into $\mathbb{R}_+^{n \times n}$ or equivalently $\mathcal{R}_Q(W_1, ... , W_M)$ is positive semi-definite for every collection \{$W_i(t)$, $i=1,2,...,M$\} such that $W_1 \leq 0$ and $W_i \geq 0$ for $i=2,\dots,M$, and for any $Q \geq 0$.
Suppose that all matrices are diagonal, thus, they commute and we can simplify and bound $\mathcal{R}_Q(W_1, ... , W_M)$ from below as follows
\begin{align} 
\mathcal{R}_Q = & Q+ H_1 W_1^2 - \sum_{i=2}^M H_i W_i^2 - 2 \sum_{j=1}^M H_j W_1 W_j \nonumber \\
	& + 2 \sum_{i=2}^M \sum_{j=1}^M H_j W_i W_j \nonumber\\
   =& Q- H_1 W_1^2 - \sum_{i=2}^M H_i W_i^2  + 2 \sum_{j=2}^M (H_1 - H_j) W_1 W_j \nonumber \\
    & + 2 \sum_{i=2}^M \sum_{j=2}^M H_j W_i W_j \nonumber\\
   =& Q- H_1 W_1^2 + \sum_{i=2}^M H_i W_i^2  + 2 \sum_{j=2}^M (H_1 - H_j) W_1 W_j \nonumber \\
	& + 2 \sum_{\substack{i,j=2\\ i \neq j}}^M H_j W_i W_j \nonumber\\
\geq & - H_1 W_1^2 + \sum_{i=2}^M H_i W_i^2 \label{ineq:s1} + 2 \sum_{\substack{i,j=2\\ i \neq j}}^M H_j W_i W_j\\
\geq & H_1 \Big( - W_1^2 + \sum_{i=2}^M W_i^2 + \sum_{\substack{i,j=2\\ i \neq j}}^M 2 W_i W_j\Big) \label{ineq:s2}\\
   = & H_1 \Big( - W_1^2 + \Big(\sum_{i=2}^M W_i \Big)^2 \Big) \nonumber  \\
   = & H_1 \Big( - W_1 + \sum_{i=2}^M W_i \Big) \Big( W_1 + \sum_{i=2}^M W_i \Big)  \label{ineq:s4} 
\end{align}
First, Inequalities (\ref{ineq:s1}) and (\ref{ineq:s2}) are due to $H_1 \leq H_i$. Second, all given matrices are assumed to be diagonal which result in all $S_i$ (solutions of (\ref{eqn:Nriccati})) to be diagonal and has the same definiteness as $W_i$ by Theorem \ref{lemma:Sdef}.
Third, $H_1 \geq 0$, Theorem \ref{lemma:Sdef} and Lemma \ref{lemma:Ssum} together yield that $ S_1 + \sum_{i=2}^M S_i \geq 0$ and therefore by (\ref{ineq:s4}) it follows that $\mathcal{R}_Q(S_1,\dots,S_M) \geq 0$.
Finally, 
similar argument as in Theorem \ref{theo:totalexistence} results in existence of $S_i$ for all $t \in [t_0,t_f]$.

\begin{IEEEbiography}[{\includegraphics[width=1in,height=1.25in,clip,keepaspectratio]{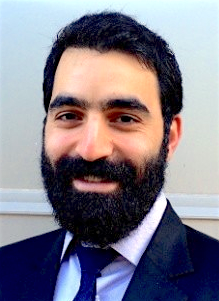}}]{Shahriar Talebi}
(S'17) received his B.Sc. degree in Electrical Engineering from Sharif University of Technology, Iran, in 2013, M.Sc. degree in Electrical Engineering from
University of Central Florida, Orlando, FL, in 2017, both in the area of control theory. He is currently working towards a PhD degree. His research interest includes Cooperative/Non-cooperative Game Theory, Networked Control, Optimal Control, Optimization and Inverse Problems.
\end{IEEEbiography}

\begin{IEEEbiography}[{\includegraphics[width=1in,height=1.25in,clip,keepaspectratio]{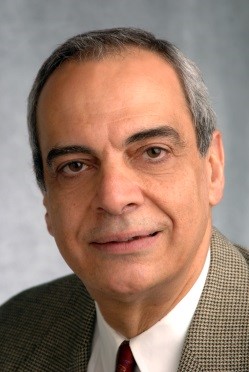}}]{Marwan A. Simaan}
(S'69-M'72-SM'79-F'88-LF'12) received the Ph.D. degree in electrical engineering from the University of Illinois at Urbana-Champaign in 1972. He is currently the Florida 21st Century Chair and Distinguished Professor of Electrical Engineering and Computer Science with the University of Central Florida. His research covers a broad spectrum of topics in game theory, control, optimization, and signal processing. He has authored or co-authored 5 books (one co-authored and 4 edited/co-edited), 140 archival journal papers and book chapters, 235 papers in conference proceedings, and 24 industry technical reports. He is a member of the U.S. National Academy of Engineering. He is a fellow of the American Society for Engineering Education, the American Association for the Advancement of Science, the American Institute for Medical and Biological Engineering, and the National Academy of Inventors. He has served on numerous professional committees and editorial boards, including the AACC Awards Committee, the IEEE Fellow Committee, the IEEE Education Medal Committee, the IEEE Proceedings and IEEE Access Editorial Boards, the AAAS Engineering Section Steering Group, and Others. In 1995, he was named a Distinguished Alumnus of the Department of Electrical and Computer Engineering at the University of Illinois at Urbana-Champaign. In 2008, he received the University of Illinois, College of Engineering Award for Distinguished Service in Engineering. He is a registered Professional Engineer with the Commonwealth of Pennsylvania.
\end{IEEEbiography}

\begin{IEEEbiography}[{\includegraphics[width=1in,height=1.25in,clip,keepaspectratio]{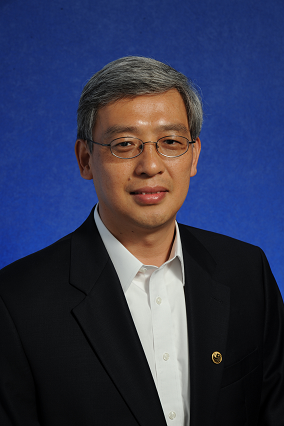}}]{Zhihua Qu}
(M'90-SM'93-F'09) received the Ph.D. degree in electrical engineering from the Georgia Institute of Technology, Atlanta, in June 1990. Since then, he has been with the University of Central Florida (UCF), Orlando. Currently, he is the SAIC Endowed Professor in College of Engineering and Computer Science, a Pegasus Professor and the Chair of Electrical and Computer Engineering, and the Director of FEEDER Center. His areas of expertise are nonlinear systems and control, with applications to autonomous systems and energy/power systems. His recent work focuses upon cooperative control, distributed optimization, and plug-and-play control of networked systems. 
\end{IEEEbiography}

\end{document}